\documentclass[runningheads]{llncs}

\usepackage{graphicx}

\usepackage{color,amsmath,amssymb}

\newenvironment{proof*}[1]
  {\renewcommand{\proofname}{#1}%
   \begin{proof}}
  {\end{proof}}

\begin{document}

\title{Sequences of symmetry groups of infinite words\thanks{Supported by Russian Foundation of Basic Research (grant 20-01-00488).}}

\author{Sergey Luchinin\inst{1}, Svetlana Puzynina\inst{1,2}}

\authorrunning{}

\institute{Saint Petersburg State University, Russia \\
\and Sobolev Institute of Mathematics, Russia\\
\email{\{serg2000lrambler.ru,s.puzynina\}@gmail.com}
}

\maketitle              

\begin{abstract}
In this paper we introduce a new notion of a sequence of symmetry groups of an infinite word. Given a subgroup $G_n$ of the symmetric group $S_n$, it acts on the set of finite words of length $n$ by permutation. We associate to an infinite word $w$ a sequence $(G_n(w))_{n\geq 1}$ of its symmetry groups: For each $n$, a symmetry group of $w$ is a subgroup $G_n(w)$ of the symmetric group $S_n$ such that $g(v)$ is a factor of $w$ for each permutation $g \in G_n(w)$ and each factor $v$ of length $n$ of $w$. We study general properties of the symmetry groups of infinite words and characterize the sequences of symmetry groups of several families of infinite words. We show that for each subgroup $G$ of $S_n$ there exists an infinite word $w$ with $G_n(w)=G$. On the other hand, the structure of possible sequences $(G_n(w))_{n\geq 1}$ is quite restrictive: we show that they cannot contain for each order $n$ certain cycles, transpositions and some other permutations. The sequences of symmetry groups can also characterize a generalized periodicity property. We prove that symmetry groups of Sturmian words and more generally Arnoux-Rauzy words are of order two for large enough $n$; on the other hand, symmetry groups of certain Toeplitz words have exponential growth. 

\end{abstract}

\section{Introduction} 
In this paper we introduce and study a new notion of a symmetry group of an infinite word. This group is related to certain algebraic and combinatorial properties of the language of factors of an infinite word, and is defined as follows. A permutation of order $n$ acts on a word of length $n$ in a natural way by permuting its letters. For every integer $n$, the symmetry group of order $n$ of an infinite word $w$ is defined as the set of permutations of order $n$ that map every factor of length $n$ of $w$ to a factor of $w$. It is easy to verify that this set indeed forms a subgroup of a symmetric group of order $n$. So, we associate with an infinite word a sequence $(G_n)_{n\geq 1}$ of subgroups of $S_n$ characterizing symmetries of the language of  factors of the infinite word. A related concept of group complexity of infinite words has been studied in \cite{CPZ17}.

We investigate general properties of symmetry groups of infinite words. We show that any subgroup of the symmetry group is a group of symmetries of some infinite word. However, the conditions for a sequence $(G_n)_{n\geq 1}$ to be a sequence of symmetry groups of an infinite word seem to be quite restrictive. We provide a series of conditions of this kind: for example,  $(G_n)_{n\geq 1}$ cannot contain certain cycles and transpositions infinitely often, unless the word is universal (i.e., contains all words as its factors), and so $G_n = S_n$ for each $n$. 
The sequences of symmetry groups also give a characterization of  generalized periodicity called periodicity by projection, where in each infinite progression with the difference equal to the period we have either a unary word or a universal word over some subalphabet. 

The symmetry groups of certain infinite words and classes of infinite words are studied. In particular, we characterize the symmetry groups of Sturmian words \cite[Chapter~2]{Lothaire} and more generally Arnoux-Rauzy words \cite{episturmian}. Sturmian words can be defined as infinite aperiodic
words which have $n+1$ distinct factors for each length $n$. They admit various characterizations in combinatorial, algebraic and arithmetical form. Their most natural generalization to non-binary alphabet, Arnoux-Rauzy words, share most of the structural properties of Sturmian words. We show that for Sturmian words and more generally Arnoux-Rauzy words the symmetry groups are small: they are of order 2 for sufficiently large $n$. The same situation holds for the Thue-Morse word \cite{Thue}.

Another family of words considered in the paper, Toeplitz words, can be defined iteratively as follows. Take an infinite periodic word $w$  on the alphabet $\Sigma \cup \{?\}$, where $?$ corresponds to a
hole. Fill the holes iteratively by substituting the word itself into the remaining
holes. In the limit, all holes are filled and an infinite word is defined. The most studied Toeplitz words are the paperfolding word  \cite{paperfolding}, giving rise to the famous fractal curve, and the period-doubling word \cite{period-doubling}. We obtain recurrent formulas for the symmetry groups of these words, showing that contrary to Arnoux-Rauzy and the Thue-Morse words, their symmetry groups are large; more precisely, they have exponential growth order.

 The paper is organized as follows. In Section 2, we fix some notation and introduce the notion of a symmetry group. We prove that the symmetry group of an infinite word is indeed a group, and show that any permutation group is a symmetry group of some word. In Section \ref{sec:sequencies}, we provide a series of conditions on the sequence $(G_n)_{n\geq 1}$ implying that the language of the underlying word contains all finite words. We also show that the sequences of symmetry groups give a characterization of a generalized periodicity property.
In Sections 4--7, we describe symmetry groups of Arnoux-Rauzy words,  the Thue-Morse word, the period-doubling word and the paperfolding word. In Section 8, we study symmetry groups of a subclass of Toeplitz words. We finish with some open problems.

Some of the results of the paper have been presented at Developments in Languages Theory 2021 conference \cite{LP2021}. This paper contains complete proofs of all results omitted in the conference version, as well as some new general properties of sequences of symmetry groups, in particular, related to periodicity of infinite words.

\section{Preliminaries}\label{sec:preliminaries}

An \textit{alphabet} $\Sigma$ is a finite set of letters. A \textit{word} on $\Sigma$ is a finite or infinite sequence of letters from $\Sigma$, i.e., $u =u_1u_2 \cdots$, where $u_i \in \Sigma$. We let $\Sigma^*$ denote the set of finite words on  $\Sigma$. The \textit{length} of the finite word $u=u_1\cdots u_n$ is the number of its letters: $|u|=n$. An empty word is denoted by $\varepsilon$, and we set $|\varepsilon|=0$. We let $u^R$ denote the \emph{reversed word} $u$: $u^R = u_nu_{n-1}\cdots u_1$. A Parikh vector of a finite word $u$ is a vector of length $|\Sigma|$ which has for each letter $a$ the number of occurrences of the letter $a$ in $u$ as its $a$'s coordinate.

A finite word $u=u_1 \cdots u_n$ is a \emph{factor} of a finite or infinite word $w$ if there exists $i \in \mathbb{N}$ such that $u= w_i\cdots w_{i+n-1}$.  We let $F(w)$ denote the
set of factors of $w$ and $F_n(w)$ the set of its factors of length $n$. 
A factor $u$ of a finite or infinite word $w$ is said to be \emph{left special} (resp., \emph{right special}) in $w$ if there exists at least two distinct letters $a, b$ such that $au$ and $bu$ (resp., $ua$, $ub$) are factors of $w$. An infinite word is called \emph{recurrent} if each of its factors occurs infinitely often. An infinite word $w$ is called \emph{universal} over an alphabet $\Sigma$ if $F(w)=\Sigma^*$.

Let $S_n$ be a group of permutations on $n$ elements. We use the following notation: $
[m_1,m_2, \ldots, m_n]=\left(\begin{array}{cccc} 1 & 2 & \dots & n \\ m_1 & m_2 & \dots & m_n \end{array}\right).$ This should be distinguished from the notation $(m_1m_2 \ldots m_k)$, which means the cycle where the element $m_i$ takes the place of $m_{i+1}$ for each $i$. 
We say that a cycle is  \emph{independent} in a permutation $g \in S_n$ if it is one of the cycles in the decomposition of $g$ into disjoint cycles.

Let $u = u_1u_2 \cdots u_n $ be a word of length $n$, and $ g \in S_n $  be a permutation on $n$ elements. Then the permutation $g$ acts on the word $u$ in the following way: $ g(u) = u_{g^{-1} (1)}u_{g^{-1} (2)}\ldots u_{g^{-1}(n)}$.

\begin{definition} 
The symmetry group $G_n(w)$ of an infinite word $w$ of order $n$ is defined as follows:
$$G_{n}(w) = \{\pi \in S_n \mid \pi(x)\in F_n(w) \,\, \text{for each} \,\, x\in F_n(w)  \}.$$
\end{definition}

In what follows when it is clear for which word this symmetric group is considered, we omit the argument and write $G_n$ for brevity.

\begin{proposition} 
Let $w$ be an infinite word, then $G_n(w)$ is a subgroup of $S_n$.
\end{proposition}

\begin{proof} 
For the proof it is enough to check the three properties of groups: 

\begin{enumerate}

\item  Let $g,h \in G_n$, then $hg \in G_n$. 
Indeed, let $x$ be a factor of $w$. By definition $G_n$, we have that $g(x)$ is a factor of $w$. So $h(g(x))=hg(x)$ is a factor of $w$. 

\item Let $e \in S_n$ be the identity permutation. Then $e \in G_n$. 
For each finite word $x$, we have $e(x) = x$. So if $x$ is a factor of $w$,  then $e(x)$ is a factor of $w$. 

\item For $g\in G_n$, we have $g^{-1} \in G_n$. 
Since the group $S_n$ is finite, there is an integer $k$ that $g^k = 1$. Then $g^{k-1} = g^{-1}$ and from item 1 we know that $g^{k-1}(x)$ is a factor of $w$. Thus $g^{-1}(x)$ is a factor of $w$.

\end{enumerate}

\end{proof}

\begin{example} 
The symmetry group $G_n$ of a universal word is $S_n$ for each $n$. Indeed, for each $g \in S_n$ and each $v \in \Sigma^n$ we have $g(v)\in F_n(w)$, since all finite words are factors. In particular, this holds for $|\Sigma|=1$: $g(a^n) = a^n$ is a factor of $a^{\mathbb{N}}$.
\end{example}

\begin{proposition} 
For each $n$ and each subgroup $G \leq S_n$, there exists an infinite word $w$ with the symmetry group $G_n(w)=G$. 
\end{proposition}

\begin{proof} 
We build $w$ on an $(n+1)$-letter alphabet $\{a_1,a_2, \ldots,a_{n+1}\}$ as follows. Set $v_i = g_i(a_1a_2\ldots a_n)$ for each $g_i \in G$, $u = a_{n+1}^n$, and $w_1$, $w_2$, \dots, $w_j$  are all possible words of length $n$  that contain the letter $a_{n+1}$. Then $w$ is defined by
$$w = u\, v_1\, u\, v_2\, u\, \cdots\, u\, v_i\, u\, w_1\, u\, w_2\, u\, \cdots u\, w_j\, u\, u\, u\, \cdots. $$  Let $x$ be a factor of $w$ of length $n$ and $g \in S_n$. Let $g \in G$. If $x$ contains $a_{n+1}$, then $g(x)$ contains $a_{n+1}$, thus $g(x)$ is a factor of $w$. If $x = g_i(a_1a_2\ldots a_n), g_i \in G$, then $gg_i \in G$ and $g(x) = gg_i(a_1a_2\ldots a_n)$ is a factor of $w$. So if $g \in G$, then $g \in G_n$. If $g \notin G$, then $g(a_1 a_2\ldots a_n)$ is not a factor of $w$. So $G_n = G$.
\end{proof}

\begin{remark} 
In the previous proposition the size of the alphabet depends on $n$ and on the subgroup $G < S_n$, since for some pairs $(|\Sigma|, G)$ words on $\Sigma$ with $G_n(w)=G$ do not exist. 

For example, on a unary alphabet we always have $G_n = S_n$. For a binary alphabet, the group $G \leq S_3$  generated by the cycle $g = (123)$ cannot be a symmetry group of a binary infinite word $w$.  
Indeed, using the fact that for each  $x\in F_3(w)$ we have $g(x), g^2(x) \in  F_3(w)$, we can show that for any  $h \in S_3$ we have $h(x)\in F_3(w)$. For $x = 000$ or $x = 111$ the statement is clear. If $x = 001$, then $w$ contains the factors $010$ and $100$, which means that $h(x)$ is also a factor. Similarly if $x = 110$. 

\end{remark}

\section{Sequences of symmetry groups} \label{sec:sequencies}

As we showed in the previous section, each subgroup of $S_n$ is a symmetry group of some infinite word. However, there are lots of restrictions on the structure of the sequence $ (G_n(w))_{n\geq 1}$ of symmetry groups. For example, we have the following:

\begin{proposition} \label{pr:cycle}
If for each $n$ a symmetry group $G_n(w)$ of a word $w$ contains the cycle $(12\dots n)$, then $w$ is universal. In particular, $G_n(w)=S_n$ for each $n$. 
\end{proposition}

\begin{proof} First notice that each letter of the word occurs in it infinitely many times. Indeed, suppose that a letter $a$ occurs for the last time at some position $i$. Then for $n>i$ the cycle $(12\dots n)$ maps the prefix of length $n$ of $w$ to a factor with $a$ at position $i+1$. Any occurrence of this factor in $w$ gives an occurrence of $a$ at a position greater than $i$. 

To prove the universality, we show that each factor of $w$ is left special and that it can be extended by any letter. Take any factor $v$ and consider any occurrence of $v$ in a word; denote its position by $i$, and consider an occurrence $j$ of any letter $a$, such that $j>i+|v|$. Now applying the cycle $(1 \dots j-i+1)$ to the factor $w_i \cdots w_{j}$ gives a word with the prefix $av$. Hence each factor of $w$ is left special and can be extended by any letter, thus $w$ is universal.  \end{proof}

More generally, the following holds:

\begin{proposition} \label{pr:cycle1}
Let $w$ be a recurrent infinite word. If for each $m$ there exist $n\geq m$ and $i\leq n-m$ such that $G_n(w)$ has an element $\sigma$ that contains an independent cycle $((i+1) (i+2)\dots (i+m))$, then $w$ is universal. 
\end{proposition}

\begin{proof} 
We will prove that in fact for each $m$ the group $G_m(w)$ contains the cycle $(1 \dots m)$. Indeed, given $m$, consider any factor $v\in F(w)$ of length $m$ and any its occurrence $l>i$, where $i$ is as in the statement. Now since $G_n$ contains a permutation $\sigma$ that contains an independent cycle $((i+1) (i+2)\dots (i+m))$, we apply this permutation to the factor $w_{l-i+1} \cdots w_{l-i+n}$. The word $\sigma (w_{l-i+1} \cdots w_{l-i+n})$ is factor of $w$, and it contains $v_2 \cdots v_m v_1$ at position $i$. Hence  $G_m(w)$ contains the cycle $(1 \dots m)$, and by Proposition \ref{pr:cycle} the word $w$ is universal.
\end{proof}

\begin{remark} For non-recurrent words there exist counterexamples, e.g., symmetry groups of $a^5b^{\infty}$  contain cycles $((i+1)\ldots(i+m))$ for $i\geq5$, although the word is not universal. It contains a universal tail though, and actually this is the general case: an infinite word with this property has a universal tail on $\Sigma'\subseteq \Sigma$. \end{remark}

Let $n$, $m$ be integers with $n\geq m$, and $\sigma\in S_n$ a cycle of length $m$: $\sigma =(j_1 j_2 \ldots j_m)$, i.e., $j_k$ are pairwise distinct. We say that $\sigma$ is \emph{arithmetical} if there exists an integer $l\geq 2$ such that all $j_k$ are congruent modulo $l$: $j_1 \equiv j_2 \equiv \ldots \equiv j_m \pmod{l}$. For example, $(137)$ is arithmetical with $l=2$, while $(134)$ is not.

\begin{proposition} Let $w$ be a recurrent infinite word, and $\sigma =(j_1 j_2 \ldots j_m)$ a cycle of length $m$ which is not arithmetical. If there exist increasing sequences of integers $(n_k)_{k\geq 1}$ and $(i_k)_{k\geq 1}$ such that there exists $\beta_{n_k}\in G_{n_k}(w)$ with $\beta_{n_k}(i_k+j_1)=i_k+j_{2}$, $\beta_{n_k}(i_k+j_2)=i_k+j_{3}$, \dots, $\beta_{n_k}(i_k+j_m)=i_k+j_{1}$, and  $\beta_{n_k}(t)=t$ for each $t=1,\dots,i_k+\min\limits_{1\leq l\leq m} j_l-1$, then the word $w$ is universal. 
\end{proposition}

\begin{proof} As in Proposition \ref{pr:cycle}, we will prove that each factor is right special and can be extended in all possible ways. Without loss of generality assume that $j_1=\min\limits_{1\leq k \leq m}{j_k}=1$.

Consider a factor $v$ of $w$. Choose $k$ with $i_k>|v|$, and consider a position $i$ of $v$ in $w$ with $i+|v|\geq i_k$. Now consider a factor $u=w_{i+|v|-i_k} \cdots w_{i+|v|-i_k+n_k-1}$. If a letter $a$ occurs at one of the positions $i+|v|+j_l-1$ for some $l=1,\dots, m$, then applying $(m-l)$ times  $\beta_{n_k}$ to $u$, we obtain $va$ as a prefix of 
$\beta_{n_k}^{m-l}(u)$.

If a letter $a$ does not occur at one of the positions $i+|v|+j_l-1$ for some $l=1,\dots, m$, then consider some its occurrence at a position $j\geq i+|v|+\max\limits_{1\leq l \leq m} j_l$. The idea is that we are going to move this $a$ to the left to the position $i+|v|+\max\limits_{1\leq l \leq m} j_l-1$ using  permutations $\beta_{n_k}$ at different positions and keeping $v$ untouched. Now we order $j_l$ in increasing order and we let $d_1, \dots, d_{m-1}$ denote the distances between consecutive $i_l$ in this order, so that the sets $\{j_1, j_2, \dots, j_m\}$ and $\{j_1, j_1+d_1, j_1+d_1+d_2, \dots, j_1+d_1+\ldots+d_{m-1}\}$ coincide. Since the cycle $\sigma$ is not arithmetical, we have that $d_l$'s are relatively prime: $(d_1,\dots, d_{m-1})=1$. 
 
First notice that for each $l$ we can move the letter $a$ by $d_l$ to the right or to the left, keeping $v$ untouched. We let $s$ and $t$ denote indices such that $d_l=j_s-j_t$. To move it to the right, we can apply  $\beta_{n_k}^{t-s}$ to $(w_{j-i_k-j_t-1}\cdots w_{j-i_k-j_t-1+n_k})$. Here we must take $k$ big enough to have $j-i_k-j_t<i$ to keep $v$ untouched. If $j-i_k-j_t<0$, then we can either consider a later occurrence of $w_i\cdots w_j$ in the word, or extend it to the left to biinfinite word keeping its set of factors (we can do it since $w$ is recurrent). Moving $a$ by $d_l$ to the left is symmetric.
 
Since $d_l$'s are relatively prime, there exist $s_1,\dots,s_{m-1}\in\mathbb{Z}$ such that $s_1 d_1+\ldots+ s_{m-1}d_{m-1}=-1$. Now applying the procedure from the previous paragraph first for $l$'s for which $s_l$'s are positive, $s_l$ times for $d_l$, then for $l$'s for which $s_l$'s are negative. This way we moved the letter $a$ to the left by 1. We now repeat the procedure until $a$ is at position $i+|v|+\max\limits_{1\leq l \leq m} j_l-1$, reducing to the above case when $a$ occurs at of the positions $i+|v|+j_l-1$.
 \end{proof}
 
As an example, the above propositions give the following corollary:

\begin{corollary} 
If for infinitely many $n$ the symmetry group $G_n(w)$ contains a transposition $(i (i+1))$, then $w$ is universal. 
\end{corollary}

\begin{proof}For the proof, we take in Proposition \ref{pr:cycle1} the permutation $\sigma$ to be a transposition $(12)$, and we apply the proposition either as stated, or in reversed form. \end{proof}

\begin{remark}
Most of results of this section can be slightly generalized, for example, for shifts, or reversed statements from right to left, or for subpermutations of a bit more general form. Also, they can be adapted to arithmetic progressions $k\mathbb{N}+i$, so that each arithmetic subword $w_iw_{i+k}\cdots$ is universal (probably with smaller alphabets on some progressions). For example, Proposition \ref{pr:cycle} can be reformulated as follows: Let $w$ be a recurrent infinite word and $k$ be an integer. If for each $n$ the symmetry group $G_n(w)$ of the word $w$ contains the cycle $(1(1+k)(1+2k)\dots)$, then each of the arithmetical subwords $w_{i+k}w_{i+2k}w_{i+3k}\cdots$, $1\leq i\leq k$, is universal over an alphabet $\Sigma_i\subseteq\Sigma$. 
\end{remark}

Now we are going to describe symmetry groups of periodic words. An infinite word $w$ is purely \emph{periodic} if there exists an integer $T$ such that $w_{n+T}=w_n$ for each $n$. An integer $T$ is called a period of of $w$. The following proposition gives a necessary condition on the sequence of symmetry groups of a word for periodicity. We then show that it is in a certain sense sufficient.

\begin{proposition}\label{pr:period}
Let $w$ be a periodic word with its minimal period $T>1$. Then for each $i\leq T$ and for each $n$ the group $G_n(w)$ contains any permutation on the indices $i, i+T,\ldots, i+kT$, where $i+kT\leq n$. We also have $G_n(w)\neq S_n$ for $n\geq T+1$. 
\end{proposition}

\begin{proof} Due to $T$-periodicity, in each factor all the letters at indices $i, i+T,\ldots, i+kT$ are the same, so a permutation on these indices does not change the factor. 
Now, since $T$ is the smallest period, we have for each factor of length $n\geq T+1$ two distinct letters. Suppose that $G_n(w)= S_n$. Then a permutation putting these two elements at positions $1$ and $T+1$ gives a word which is not a factor of $w$. A contradiction. 
\end{proof}

We say that a word $w$ is \emph{periodic by projection} if there exists an integer $T$ such that for each $i\in \{0, \dots, T-1\}$ the word $w_iw_{i+T}w_{i+2T}\cdots$ is universal over some subalphabet $\Sigma_i\subset \Sigma$. Clearly, if each  $\Sigma_i$ is of cardinality 1, then the word is periodic. We say that a periodic by projection word $w$ is \emph{mixing} if for each $n$ and all subwords $w_{i_j}w_{i_j+T}\cdots w_{i_j+(n-1)T}$, $i_j \equiv j \mod T$, one has $w_{i_1}w_{i_2} \cdots w_{i_T} w_{i_1+T}w_{i_2+T} \cdots w_{i_T+T} \cdots w_{i_1+(n-1)T}w_{i_2+(n-1)T} \cdots w_{i_T+(n-1)T} $ a subword of $w$. Clearly, a periodic word is mixing. 
The condition from the previous proposition gives a necessary and sufficient condition for a word to be periodic by projection and mixing:

\begin{proposition}\label{pr:period1}
A word $w$ is periodic by projection with the period $T$
and is mixing if and only if for each $i\leq T$ and for each $n$ the group $G_n(w)$ contains any permutation on the indices $i, i+T,\ldots, i+kT$, where $i+kT\leq n$.  
\end{proposition}

\begin{proof} The properties of the sequences of symmetry groups of words which are periodic by projection and mixing are proved similarly to the proof of Proposition \ref{pr:period}. Now we prove the converse. Since for each $i\leq T$ and for each $n$ the group $G_n(w)$ contains any permutation on the indices $i, i+T,\ldots, i+kT$, where $i+kT\leq n$, we have that the word $w$ contains all the permutations of arithmetical subwords with difference $T$. 
Suppose that at indices congruent to $i_1$ modulo $T$ the word contains all symbols from a subalphabet $\Sigma'\subseteq \Sigma$. To prove that the subword in this arithmetic progression is universal over $\Sigma'$, we need to prove that each Parikh vector occurs. For that, it is enough to prove that we have infinitely many occurrences of each letter from $\Sigma'$ at indices congruent to $i_1$ modulo $T$. Let $a$ be a letter in $\Sigma'$ which occurs at a position $j_1\equiv i_1 (\mod T)$. Let $i_1, \dots, i_k$ be all indices mod $T$ such that corresponding arithmetic $T$-progressions contain the letter $a$, at positions $j_1, j_2, \ldots, j_k$, respectively. Consider a factor $u=w_0 w_1 \cdots w_n$ with $n\geq \max j_l+T$; we can take a permutation $\sigma$ which sends each index $j_l$ to an index greater than $\max j_l+T$ in the same arithmetical progression. So, the new factor $\sigma(u)$ contains a letter $a$ at a position bigger than the first one chosen, in each arithmetic progression. Wherever this factor occurs in $w$, it gives gives another occurrence of the letter $a$ at each arithmetic progression $(j_l+mT)_{m\in\mathbb{N}}$ at a position greater than $j_l$ (here we use the fact the the word is one-way infinite). In the same way we get the third occurrence of $a$ in each arithmetic progression $(j_l+mT)_{m\in\mathbb{N}}$. Continuing this line of reasoning, we obtain that each letter occurs infinitely many times in the corresponding arithmetic progression. To finish the proof, it is enough to notice that to get a given Parikh vector, we should take a factor long enough to contain enough occurrences of each letter, and then apply a permutation which sends them to the beginning. By applying all arithmetic permutations, we get universality. Mixing property is straightforward. 
\end{proof}

The above proposition gives a characterization of a generalized periodicity property (in the sense of periodicity by projection and mixing). In addition, it gives yet some more restrictions on the structure of the symmetry groups of an infinite word.

\begin{corollary} Let $T$ be an integer and $w$ be a word such that for each $j\leq T$ and for each $n$ the group $G_n(w)$ contains any permutation on the indices $i, i+T,\ldots, i+kT$, where $i+kT\leq n$.    Then for each $n=mT$, where $m$ is an integer, the group $G_n(w)$ contains a cyclic shift $(12\ldots n)$.\end{corollary}

\begin{proof} Proposition \ref{pr:period1} implies that the word is periodic by projection and mixing, and such word contains cyclic shifts for lengths divisible by the period. Indeed, for $n=Tm$, consider a factor $u$ of $w$ at a position $j$. If the subword of $w$ at the positions congruent to $j$ modulo $T$ is unary, then the cyclic shift of $u$ is the factor of $w$ at the position $j+1$. If it is universal, then it is either the factor at the position $j+1$, or there exists another position with such factor due to mixing property. \end{proof}

\begin{remark} We remark that symmetry groups do not allow to distinguish between normal periodicity and periodicity projection due to a similar behaviour on universal arithmetic subwords on unary and non-unary alphabets. \end{remark}

\begin{remark} Proposition \ref{pr:period1} does not hold for biinfinite words. For example, a word $\cdots ababababbbabababab\cdots$ has indicated subgroups, but it is not periodic by projection. \end{remark}

\section{The symmetry groups of Arnoux-Rauzy words}

In this section, we show that starting from some length symmetry groups of Arnoux-Rauzy words contain only two elements, identity and mirror image.

\begin{definition} 
An infinite word $w$ on an alphabet $\Sigma$ is called an Arnoux-Rauzy word if the following conditions hold:

$\bullet$ if a finite word $u$ is a factor of $w$, then $u^R$ is a factor of $w$

$\bullet$ $w$ has exactly one left special factor (or, equivalently, right special factor) of each length. Moreover, each special factor extends in exactly $|\Sigma|$ ways.
\end{definition}

We denote the left special factor of length $n$ by $U_n^l$ and the right one by $U_n^r.$ It follows from the definition that $U_n^r = (U_n^l)^R$.

\begin{definition} \cite{episturmian} 
The palindromic right-closure $w^{(+)}$ of a finite word $w$ is the (unique) shortest palindrome having $w$ as a prefix. That is, $w^{(+)} = wv^{-1}w^R$, where $v$ is the longest palindromic suffix of $w$. 
The iterated palindromic closure function,  denoted by $Pal$, is defined recursively as follows. Set $Pal(\varepsilon) = \varepsilon$, and for any word $w$ and letter $x$ define $Pal(wx) = (Pal(w)x)^{(+)}$.
\end{definition}

\begin{theorem} \cite{episturmian} 
A word $w$ is an Arnoux-Rauzy word if and only if there exists an infinite word $v$ with each letter occurring in it infinitely many times such that $w$ has the same set of factors as the limit of iterated palindromic closures of prefixes of the word $v$. 
\end{theorem}

Let $w$ be an Arnoux-Rauzy word with the language of the palindromic closure of $v$. Let $v$ begin with $t$ identical letters; then we denote the number $t$ by $A(w)$.

\begin{theorem} 
Let $w$ be an Arnoux-Rauzy word and $t = A(w)$. Then for $n \geq 4t + 2$ the symmetry group of $w$ is $$G_n = \{[1,2,\ldots,n-1,n],  [n,n-1,\ldots,2,1]\}.$$
\end{theorem}

\begin{proof} 
From the definition of Arnoux-Rauzy words it follows that 
$\{[1,2,\ldots,n-1,n],  [n,n-1,\ldots,2,1]\} \subset G_n$. It remains to prove that any other permutation cannot belong to $G_n$.
Let be $\sigma \in G_n$. We first show that $\sigma(n) = 1$ or $\sigma (n) = n$ for $n \geq 2t+1$. Suppose that $\sigma (n) = k$,  for some $1 < k < n$.

We let $a$ denote the letter which is the special factor of length 1. Suppose first that $k < n-t$. Then consider all possible factors $U_{n-1}^rx$, where $x$ is any letter. Then applying $\sigma$ to all such factors, for some $u\in F_{k-1}(w)$, $v\in F_{n-k}(w)$ and any $x\in\Sigma$, we get $|\Sigma|$ factors of the form $uxv$, which differ from each other only at position $k$. Since $a$ is the only left and right special factor of length 1, then the letter preceding $x$ and the letter following $x$ is $a$. So, we get factors of the form $u'axav'$, where $x$ is any letter. Suppose that $v$ starts with $a^r$ and then there is another letter $b\neq a$. Then we get the factors $u'axa^rbv''$. Now if we take $x = a$, we see that the factor $a^{r+1}$ can continue to the right with the letters $a$ and $b$, and the factor $a^rb$ can continue to the left with any letter ($x$ can be any letter). So we get at least two special factors, which is impossible. Thus we showed that the word $v$ can only consist of the letters $a$. But $v$ is longer than $t$ ($|v| = n - k > t$), so this is impossible ($w$ cannot contain more than $t$ consecutive letters $a$). 
If $k \geq n - t$, then we apply the same line of reasoning for the word $u$. This is possible, because $n \geq 2t+1$, so $k > t$ and $|u| \geq t$. 

So, we have $\sigma (n) = 1$ or $\sigma (n) = n$. Similarly we show that $\sigma(1) = 1$ or $\sigma (1) = n$.

Suppose that $\sigma (n) = n$, then $\sigma (1) = 1$. Then for similar reasons we get that $\sigma(2) = 2,\sigma(3) = 3, \ldots, \sigma(n-2t-1) = n - 2t - 1$. Indeed, let $U = u_1\cdots u_n$ be a factor of $w$. Then $V = u_2u_3\cdots u_n$ is a factor of $w$. This means that $\sigma'(V)$ is factor of $w$, where $\sigma'$ is the restriction of $\sigma$ to the word $V$. So, $\sigma' \in G_{n-1}$. It follows that $\sigma'(n-1) = n-1$ and $\sigma'(1) = 1$ or $\sigma'(n-1) = 1$ and $\sigma'(1) = n-1$, because $n-1 \geq 2t+1$. Since $\sigma(n) = n$, we have $\sigma'(n-1) = n-1$, so $\sigma'(1) = 1$ and $\sigma(2) = 2$. In the same way step by step we get that $\sigma(k) = k$ for all $k \leq n-2t-1$. Similarly, we can show that $\sigma(k) = k$ for any $k \geq 2t + 1$. Since $n \geq 4t+2$, we get that $\sigma(k) = k$ for any $1 \leq k \leq n$. So $\sigma = \mathrm{Id} = [1,2,\ldots,n]$.

If $\sigma (1) = n$, $\sigma (n) = 1$,  then we can prove that $\sigma(k) = n-k$ in a similar way. So, in this case we prove that $\sigma = [n,n-1,\ldots,2,1]$. The theorem is proved.
\end{proof}

Since Sturmian words are binary Arnoux-Rauzy words, we have 

\begin{corollary} 
Let $w$ be a Sturmian word and $k$ be the length of the longest block of the more frequent letter in $w$. Then $$G_n = \{[1,2,\ldots,n-1,n],  [n,n-1,\ldots,2,1]\} \quad \mbox{for each} \quad n > 4k+2.$$
\end{corollary}

\section{Symmetry groups of the Thue-Morse word}

In this section we describe the symmetry groups of the Thue-Morse word. We show that its symmetry groups are also small, as the symmetry groups of Arnoux-Rauzy words.

\begin{definition}
{\it The Thue-Morse word} is the infinite word $w = a_0a_1a_2\cdots$, where $a_n$ is the parity of the number of $1$'s in the binary expansion of $n$.
\end{definition}

\begin{theorem}\label{thm:Thue-Morse} 
The symmetry group of the Thue-Morse word is expressed as follows:
$$G_n=
\begin{cases}
S_n, & \mbox{if n} \leq 3, \\ 
\{[1,2,3,4], [4,3,2,1], [1,3,2,4], [4,2,3,1]\}, & \mbox{if n} = 4,\\ 
\{[1,2,3,4,5], [5,4,3,2,1], [2,1,3,5,4], [4,5,3,1,2] \}, & \mbox{if n} = 5,\\
\{ [1,2,3,4,5,6], [6,5,4,3,2,1], [6,2,4,3,5,1], [1,5,3,4,2,6]\}, & \mbox{if n} = 6,\\
\{ [1,2,\ldots,n-1,n], [n,n-1,\ldots,2,1] \},  & \mbox{if n} > 6. \end{cases}$$

\end{theorem}

\smallskip
In the proof of the theorem we will use the following well-known fact \cite{Thue}:

\begin{lemma} \label{lemma:Thue}
Let $w$ be the \textit{Thue-Morse} word. Then $w$ does not have a factor $uua$ where $u$ is a finite word and $a$ is its first letter.
\end{lemma}

In particular, it follows that $w$ does not contain the factors $000$, $111$, $01010$.

Let $G_n$ be the symmetry group for the Thue-Morse word. We prove the following lemma:

\begin{lemma} \label{lemma:Thue-2} Let $g \in G_n$ and $g(a) = i, g(b) = i+1$ and $g(c) = i+2$. Then the numbers $a,b,c$ form an arithmetic progression with difference $2^k$.
\end{lemma}

\begin{proof} 
In this proof we consider all numbers in binary notation. Let the positions $a,b,c$ of the factor $u$ contain the letters $a_x,a_{x+n}$ and $a_{x+m}$ of the Thue-Morse word $a_1a_2\cdots$, where $n < m$ are positive integers and $a,b,c$ are $a_{x},a_{x+n},a_{x+m}$ in some order. If $a_x = a_{x+n} = a_{x+m}$ then $g(u)$ has three 1's or three 0's in a row. So $g(u)$ cannot be a factor from Lemma \ref{lemma:Thue}. This means that this double equality is impossible. We prove that if $n$ and $m$ are not two consecutive powers of two then there is a factor $u$ such that $a_x = a_{x+n} = a_{x+m}$, which we proved is impossible. We denote the number of 1's in the binary expansion of $n$ by $f_2(n)$. Consider two cases:

\textbf{Case 1}: $f_2(m) = f_2(n) = k$. 

\textbf{1.1.} If $k = 0$ ($n$ and $m$ have an even number of 1's in the binary expansion), then $x$ can be $2^k = 1000\cdots 000$ where the number of 0's in the number $x$ is greater than the number of positions in binary expansion of $m$ and $n$. Then $f_2(x) = f_2(x+m) = f_2(x+n) = 1$, a contradiction.

\textbf{1.2.} If $k = 1$ ($n$ and $m$ have an odd number of 1's in the binary expansion), then let the largest positions of the numbers $m$ and $n$ in the binary expansion be $l_m$ and $l_n$ respectively.

\textbf{1.2.1.} If $l_m = l_n$, then we put 1 at position $l_m$ of the number $x$. We get $f_2(x) = f_2(x+n) = f_2(x+m) = 1$, a contradiction.

\textbf{1.2.2.} Let $l_m \neq l_n$. We know that $m > n$. So $l_m > l_n$. If $n = 1\cdots01\cdots$ has 1 at position $l$ and 0 at position $l + 1 < l_n$. Then we can put 1 in $x$ at position $l$ (we assume that initially $x = 0$ and we add to $x$ the number $2^l$). We get that $f_2(x) = f_2(x+n) = 1$. If $f_2(x+m) = 1$ then we get a contradiction. Otherwise we put 1 at position $l_{x+m}$ of the number $x$ (add the number $2^{l_{x+m}}$ to $x$) where $l_{x+m}$ is the highest nonzero position of the number $m+x$. Since $n$ has 0 at position $l_{x+m}$, then $f_2(x) = f_2(x+n) = f_2(x+m) = 0$ a contradiction. Similarly we can first put 1 at position $l_n$ of the number $x$ and if necessary put 1 at position $l_{x+m}$. These arguments work in all cases except $n = 11\cdots 1100\cdots 00$, $m = 10\cdots $ and $l_m = l_n + 1$ ($n$ has 1 at the first positions and 0 at the next, $m$ has $10$ in the first two positions). We consider the following subcases:

{\bf (a)} If $n$ has 3 or more 1's, then:

\textbf{(a. 1)} If $m = 100\cdots $ ($m$ has $0$ at the third position), then we can put 1 at positions $l_m$ and $l_{m-2}$  in  $x$. Then $x+n = 10001\cdots $, $x+m = 1001\cdots $ and $f_2(x) = f_2(x+n) = f_2(x+m) = 0$; a contradiction.

\textbf{(a. 2)} If $m = 101\cdots $ (third position 1), then we can put 1 on $l_m, l_{m-1}$ and $l_{m-2}$ positions in the $x$. Then $x+n = 101\cdots$, $x+m = 1100\cdots$ and $f_2(x) = f_2(x+n) = f_2(x+m) = 1$; a contradiction.

{\bf (b)} Let $n = 100\cdots00$. If $m = 100\cdots00$, then we get that $m = 2n = 2^{k+1}$. So $a, b, c$ form a progression with  difference  $2^k$. Otherwise there is $l < l_{n}$ that $m$ has 1 at position $l$ and 0 at position $l+1$. Then we put 1 at positions $l$ and $l_{n}$ in the number $x$. Then $f_2(x) = f_2(x+n) = f_2(x+m) = 0$; a contradiction. This completes the proof in Case 1.

\textbf{Case 2}: $f_2(m) \neq f_2(n)$. 

\textbf{2.1.} If $l_m > l_n + 1$, then we put 1 at position $l_m$ in $x$. Then $f_2(x+m) = f_2(x+n) = k$, $f_2(x) = 1$. If $f_2(m)$ is odd, then $k = 1$, a contradiction. Otherwise put 1 on the $l_n$ position in the $x$. Then $f_2(x) = f_2(x+n) = 0$. If $f_2(x+m) = 0$, then we get a contradiction. Otherwise put 1 at position $l_m+1$ in $x$. Then $f_2(x) = f_2(x+n) = f_2(x+m) = 1$; a contradiction.

\textbf{2.2.} If $l_m = l_n+1$, then put 1 at position $l_m$  in $x$. Then $f_2(x+m) = f_2(x+n) = k$, $f_2(x) = 1$. If $f_2(m)$ is odd, then $k = 1$, a contradiction. Otherwise if $n + x = \cdots  01\cdots $ has 0 at position $l+1$ and 1 at position $l$ for some $l < l_n$ then we put 1 at position $l$ in $x$. We get $f_2(x) = f_2(x+n) = 0$. If $f_2(x+m) = 0$ we have a contradiction. Otherwise we put 1 at position $l_m+1$ in $x$. We get $f_2(x) = f_2(x+n) = f_2(x+m) = 1$, a contradiction. It remains to get a contradiction in the case when $f_2(m) = 0$ and $n + x = 11\cdots 1100\cdots 00$ with at least two 1's since $f_2(x+n) = 0$.  

\textbf{2.2.1.} If $m+x = 100\cdots $ ($m+x$ has 0 at the third position), then consider the number $s$ which has 1 at positions $l_m$ and $l_{m-1}$. Then $s+x = 100\cdots,  s+x+n = 110\cdots $, $s+x+m = 111\cdots$ and $f_2(s+x) = f_2(s+x+n) = f_2(s+x+m) = 0$; a contradiction.

\textbf{2.2.2.} If $m+x = 101\cdots $ ($m+x$ has 1 at the third position), then consider the number $s$ which has 1 at positions $l_m+1, l_{m}$ and $l_{m}-1$. Then $s+x = 1001000\cdots,  s+x+n = 1010\cdots,$ $s+x+m = 1100\cdots $ and $f_2(s+x) = f_2(s+x+n) = f_2(s+x+m) = 0$; a contradiction.

\textbf{2.3.} If $l_m = l_n$, then let $l$ be the largest index such that $m$ and $n$ have different bits at position $l$. And let $m$ has 0 at positions $l_1 > l_2 > \ldots > l_s > l$. Put 1 at positions $l_1,l_2,\ldots,l_s$ and $l$. Let the largest positions of the numbers $x+m$ and $x+n$ in the binary expansion be $l_{x+m}$ and $l_{x+n}$ respectively. We can notice that $l_{x+m} = l_{x+n}+1$ and $x+m$ has 0 at position $l_{x+n}$. The are two equal numbers among $f_2(x), f_2(x+m)$ and $f_2(x+n)$. 

\textbf{2.3.1.} If $f_2(x) = f_2(x+n) = k$, then in the case $f_2(x+m) = k$ we get a contradiction. In the case $f_2(x+m) \neq k$ put $1$ at position $l_{x+m}$ in $x$. Then $f_2(x) = f_2(x+n) = f_2(x+m)$; a contradiction. 

\textbf{2.3.2.} If $f_2(x) = f_2(x+m) = k$, then in the case$f_2(x+n) = k$ we get a contradiction. In the case $f_2(x+n) \neq k$ put $1$ at position $l_{x+n}$ in $x$. Then $f_2(x) = f_2(x+n) = f_2(x+m)$; a contradiction. 

\textbf{2.3.3.} If $f_2(x+n) = f_2(x+m) = k$. If $f_2(x) = k$, then we get a contradiction. Otherwise put 1 at position $l_{x+m}$ and $l_{x+n}$ in $x$. Then $f_2(x) = f_2(x+n) = f_2(x+m)$; a contradiction. 

The lemma is proved. 
\end{proof}

\noindent\textit{Proof of Theorem \ref{thm:Thue-Morse}}.  
Let $g=[a_1,a_2,\ldots, a_n] \in G_n$, then $g^{-1}\in G_n$, and $g^{-1}$ is defined by $g(a_i) = i$. From Lemma \ref{lemma:Thue-2} we get that $a_{i-1}, a_i$ and $a_{i+1}$   for each $i$ form an arithmetic progression with difference $2^k$ (where $k$ could depend on $i$).
Let $n > 6$. If for each triple the difference of arithmetic progressions at least 2, then these differences are powers of 2. 

So all $a_i$ are of the same parity, which is false. Therefore, there are three elements $a_{i}, a_{i+1}, a_{i+2}$ consisting of three consecutive integers. We prove that we can chose $i$ such that $i \geq 3$ and $i+2 \leq n-1$. Suppose the converse. Then $a_3, a_4, \ldots , a_{n-1}$ are of the same parity. But it is possible only in the case $n = 7$ (otherwise the number of the elements of the same parity is more than $\frac{n+1}{2}$, which is impossible). But in the case $n = 7$ we get that $\{a_1,a_2,a_7\} = \{2,4,6\}$ and $\{a_3, a_4, a_5, a_6\} = \{1,3,5,7\}$. So, for the triples $(a_1,a_2,a_3)$, $(a_2,a_3,a_4)$, $(a_5,a_6,a_7)$ the difference of arithmetic progressions is 1. So $a_1,a_2,a_3,a_4$ are consecutive positive integers and $a_5,a_6,a_7$ are also consecutive positive integers. So either $\{a_1,a_2,a_3,a_4\} = \{1,2,3,4\}$ and $\{a_5,a_6,a_7\} = \{5,6,7\}$, or $\{a_1,a_2,a_3,a_4\} = \{4,5,6,7\}$ and $\{a_5,a_6,a_7\} = \{1,2,3\}$. In first case $\{a_1,a_2\} = \{2,4\}$, so $a_3 = 3$ and $a_4 = 1$, then $\{a_5, a_6\} = \{5,7\}$, but then $a_4,a_5$ and $a_6$ do not form an arithmetic progression. In second case $\{a_1,a_2\} = \{4,6\}$, so $a_3 = 5$ and $a_4 = 7$, then $\{a_5, a_6\} = \{1,3\}$, but then $a_4,a_5$ and $a_6$ do not form an arithmetic progression.
So for $n \geq 7$ we can choose three consecutive integers $a_{i}, a_{i+1}, a_{i+2}$, $i \geq 3$ and $i+2 \leq n-1$. Consider several cases:

\textbf{1.} If $a_{i} = k, a_{i+1} = k+1, a_{i+2} = k+2$, then $a_{i+3} = k+3$ etc. Similarly in the other direction. So, we get the identical permutation $[1,2,\ldots,n-1, n]$.

\textbf{2.} Similarly for the case $a_{i} = k+2, a_{i+1} = k+1, a_{i+2} = k$ we obtain the permutation $[n, n-1,\ldots, 2, 1]$.

\textbf{3.} If $a_{i} = k, a_{i+1} = k+2, a_{i+2} = k+1$, then $a_{i+3} = k+3$. It also follows from Lemma \ref{lemma:Thue-2} that $a_{i-1} = k-2$ or $a_{i-1} = k+4$. 

\textbf{3.1.} If $a_{i-1} = k-2$, then we get from Lemma \ref{lemma:Thue-2} that $a_{i-2} = k-1$. Then there is a factor $w$ of length $n$ which has the subword $001100$ at positions $k-2, k-1, \ldots, k+3$, because $f(2^t + 2^{t-1} + 5) = 0,  f(2^t + 2^{t-1} + 6) = 0, f(2^t + 2^{t-1} + 7) = 1, f(2^t + 2^{t-1} + 8) = 1, f(2^t + 2^{t-1} + 9) = 0,  f(2^t + 2^{t-1} + 10) = 0$ for sufficiently large $t$. But then $g(w)$ contains the sequence 01010 ($a_{i-2} = 0, a_{i-1} = 0, a_{i} = 1, a_{i+1} = 0, a_{i+2} = 1, a_{i+3} = 0$). It is impossible by Lemma \ref{lemma:Thue}.

\textbf{3.2.} If $a_{i-1} = k+4$, then Lemma \ref{lemma:Thue-2} implies that $a_{i-2} = k+8$ or $a_{i-2} = k-4$. 

In the first case consider a factor of $w$ which has a factor 001100101 at positions $k-4, k-3, \ldots, k+4$. The factor 001100101 occurs in the The-Morse word e.g. starting from the positions of the form $2^t + 2^{t-1} + 5$, for $t$ large enough. Then $a_{i-2} = 0, a_{i-1} = 1, a_i = 0, a_{i+1} = 1, a_{i+2} = 0$. It is impossible by Lemma \ref{lemma:Thue}. 

In the second case there is a factor $w$ which has a factor 001011010 at positions $k, k+1, \ldots, k+8$ (for example, at a postition of the form $2^t + 2^{t-1} + 9$ in the Thue-Morse word). Then $a_{i-2} = 0, a_{i-1} = 1, a_i = 0, a_{i+1} = 1, a_{i+2} = 0$. It is also impossible by Lemma \ref{lemma:Thue}. 

\textbf{4.} Other cases are considered similarly.

\medskip

We now prove that $G_n$ contains the permutations $g = [1,2,\ldots, n-1,n]$ and $[n,n-1,\ldots, 2,1]$. Indeed, the set factors of the Thue-Morse word is closed under reversal. 

It remains to treat small values of $n$, i.e. $n < 7$. For $n = 1,2,3$ we get $G_n = S_n$. For $n = 4,5$ from Lemma \ref{lemma:Thue-2} we get that only the following permutations are possible: $G_4 = \{[1,2,3,4], [4,3,2,1], [1,3,2,4], [4,2,3,1]\}$ $G_5 = \{ [1,2,3,4,5], [5,4,3,2,1], [2,1,3,5,4], [4,5,3,1,2]  \}$. This can be checked by computer. For $n = 6$ with some additional considerations from the general proof we get that only the following permutations are possible: $G_6 = \{[1,2,3,4,5,6], [6,5,4,3,2,1], [6,2,4,3,5,1], [1,5,3,4,2,6]\}$. This can also be checked by computer. 

The theorem is proved.

\section{Symmetry group of the period-doubling word}

In this section we characterize symmetry groups of the period-doubling word by recurrence relations and show that the size of the group grows exponentially.

\begin{definition}
\textit{The period-doubling word} $w = 0100010101000100\cdots$ is a binary word defined as follows:
$w_i \equiv k \, \pmod{2}$, where $i = 2^kl$ and $l$ is odd.
\end{definition}
In other words, $w_i$ is equal to the maximal exponent of 2 dividing $i$, modulo 2.

\begin{theorem}\label{thm:perioddoubling} For $n \geq 4$ the symmetry group $G_{n}$ of the period-doubling word satisfies the following recurrence relations:
$$
G_n= \begin{cases}
(G_{n/2}\times G_{n/2}) \rtimes S_2,  & \text{if n  is even}, \\ 
G_{(n-1)/2}\times G_{(n+1)/2}, &  \text{if n is odd},  \end{cases}
$$
and $G_1 = S_1$, $G_2 = S_2$, $G_3 = \{[1,2,3],[3,2,1]\}.$ 
\end{theorem}

\begin{remark} The groups in Theorem \ref{thm:perioddoubling} are as follows. When $n$ is odd, permutations from $G_n$ act as permutations  on  even  positions  as  permutations from $G_{(n-1)/2}$ and on odd positions as permutations from $G_{(n+1)/2}$ independently.
When $n$ is even, permutations from $G_n$ act on even and odd positions independently as $G_{n/2}$, plus we can trade places even and odd positions. So, in the formula $(G_{n/2}\times G_{n/2}) \rtimes S_2$ the permutations from the first direct product of two copies of $G_{n/2}$ do not change parity of elements and permutations from the second direct product of two copies of $G_{n/2}$ change parity of each element. 
\end{remark}

\begin{remark}
In the conference version \cite{LP2021} there was a misprint in the theorem statement with direct product instead of semidirect product. The subgroups of $S_n$ indicated in the body of the proof were correct.
\end{remark}

For the proof, we need the following lemma:

\begin{lemma} \label{lemma:period-dubling}
Let $g \in S_n.$ Suppose that there are two positions $c$ and $d$ of the same parity such that $g(c)$ is even  and $g(d)$ is odd. Then $g \not\in G_n$. 
\end{lemma}

\begin{proof} Note that for any $u\in F(w)$ either on even positions of $u$ or on odd positions of $u$ we have only $0$'s, since $w_{2k+1} = 0$ for any $k$. Next, each factor of length at least $4$ contains an occurrence of 1, because there is a number that is not a multiple of 4 among two consecutive even numbers. 

We set $d - c = 2k$; then there is an integer $s$ such that $w_s = w_{s+2k} = 1$,  which is equivalent to the fact that the maximal exponent of 2 in $s$ and $s+2k$ is odd. Indeed, if $k$ is odd, then we can take $s = 2^{2m+1}$ for some $m \in \mathbb{N}$.
Then the maximal exponent of 2 in $s$ and $s + 2k$ is equal to $2m+1$ and $1$, respectively. If $k$ is even, then we can take $s = 2m$, where $m$ is odd, since the maximal exponent of 2 in $s$ and $s + 2k$ is equal to $1$. Then we get that there is a factor $u$ which has $w_s = w_{s+d-c} = 1$ at positions $c$ and $d$, because $s$ can be chosen sufficiently large. But then $g(u)$ has even and odd positions that have 1, which is impossible. It follows that $g(u)$ is not a factor of $w$, so $g \notin G_n$. The lemma is proved. 
\end{proof}

\begin{proof*}{Proof of Theorem \ref{thm:perioddoubling}} It is easy to see that $G_1 = S_1$, $G_2 = S_2$, $G_3 = \{[1,2,3],[3,2,1]\}$.
Consider $n \geq 4$. Lemma \ref{lemma:period-dubling} implies that if $n$ is odd, then $g$ does not change the parity of positions. If $n$ is even, then $g$ either does not change the parity of positions, or changes the parity of all positions (the latter is impossible for an odd $n$ since there are more odd positions). 

Let $n = 2l+1$ be odd, then $g \in S_n$ permutes odd positions as a permutation $h \in S_{l+1}$ and even positions as a permutation $p \in S_l$. Let $u = w_{m+1}w_{m+2} \ldots w_{m+n}$ be a factor of length $n$.

If $m$ is even, then $u$ has 0 at odd positions, and hence $g(u)$ also has 0 at odd positions. There is an occurrence of $1$ in any factor of length $n\geq 4$, so $g(u)$ occurs at an odd position: $ g(u) = w_{r+1}w_{r+2}\cdots w_{r+n}$, where $r$ is even. Then $p (w_{m+2}w_{m+4}\cdots w_{m+n-1}) = w_{r+2}w_{r+4}\cdots w_{r+n-1}$. The definition of the period-doubling word implies that  $w_k = 1 - w_{2k}$. It follows that $p(w_{\frac{m+2}{2}}w_{\frac{m+4}{2}}\cdots w_{\frac{m+n-1}{2}}) = w_{\frac{r+2}{2}}w_{\frac{r+4}{2}}\cdots w_{\frac{r+n-1}{2}}$ for any even $m$. This means that $p \in G_{l}$, because otherwise we can pick $m$ such that $p(w_{\frac{m+2}{2}}w_{\frac{m+4}{2}} \ldots w_{\frac{m+n-1}{2}})$ is not a factor of $w$. These arguments also imply that if $p \in G_l$, then $g(u)$ is a factor, because there is a factor of $p(w_{\frac{m+2}{2}}w_{\frac{m+4}{2}}\cdots w_{\frac{m+n-1}{2}}) = w_{\frac{r+2}{2}}w_{\frac{r+4}{2}}\cdots w_{\frac{r+n-1}{2}}$ and hence $g(w_{m+1}w_{m+2}w_{m+3}\cdots w_{m+n}) = g(0w_{m+2}0w_{m+4} \cdots)  = 0w_{r+2}0w_{r+4}\cdots = w_{r+1}w_{r+2}\cdots w_{r+n}$.  Then $p \in G_l$. If $m$ is odd, in a similar way we get that $h \in G_{l+1}$. And the arguments above imply that if $p \in G_l$ and $h \in G_{l+1}$, then $g \in G_{2l+1}$.

As a result, we get for odd $n = 2l+1$ that $g$ acts on even positions as a permutation from $G_l$ and on odd positions as a permutation from $G_{l+1}$.

Now let $n = 2l$ be even, and first consider the case when the permutation $g \in G_{2l}$ does not change the parity of positions. Similarly to the case of odd $n$, we get that $g$ acts on even positions as a permutation $p$ from $G_l$ and on odd positions as a permutation $h$ from $G_{l}$. Now consider the case when $g$ changes the parity of all positions. For a factor $u = w_{m+1}w_{m+2}\cdots w_{m+n}$, its image occurs at some position $r+1$: $g(u) = w_{r+1}w_{r+2}\cdots w_{r+n}$. Then $h(w_{m+1}w_{m+3}\cdots w_{m+n-1}) = w_{r+2}w_{r+4}\cdots w_{r+n}$ and  $p(w_{m+2}w_{m+4}\ldots w_{m+n}) = w_{r+1}w_{r+3}\cdots w_{r+n-1}$. If $u$ has 0 at odd positions, then $g(u)$ has 0 in even positions, which means that similarly to the previous case we can get that $p \in G_l$ and  $h \in G_l$. The arguments above also imply that if $p \in G_l$ and $h \in G_{l+1}$, then $g \in G_{2l+1}$.

Summing up, we proved that permutations from $G_{2n+1}$ act as $G_n$ at even positions and as $G_{n+1}$ at odd ones independently, so that $G_{2n+1} = G_n \times G_{n+1}$. Permutations from $G_{2n}$ act as $G_n$ at even and odd positions independently, plus we can trade places even and odd positions. Therefore, the permutations from $G_{2n}$ form the group $(G_{n/2}\times G_{n/2}) \rtimes S_2$.
The theorem is proved.  
\end{proof*}

\begin{corollary} 
The size of the symmetry group of the period-doubling word satisfies the following lower bound:
$|G_n| > 2^{\frac{n}{3}} $. In particular, $|G_{2^n}| = 2^{2^n-1}$.
\end{corollary}

\begin{proof}
First we prove by induction that $|G_n| > 2^{\frac{n}{3}} $. For small $n$ we check directly that it holds. Then $|G_{2n+1}| = |G_{n}|\cdot|G_{n+1}| > 2^{\frac{n}{3}+\frac{n+1}{3}} = 2^{\frac{2n+1}{3}}$ and $|G_{2n}| = 2|G_{n}|\cdot|G_{n}| > 2^{2\frac{n}{3}+1} = 2^{\frac{2n}{3}+1}.$

Now we prove that $|G_{2^n}| = 2^{2^n-1}$. Obviously, $|G_2| = 2$. Then $|G_{2^{n+1}}| = 2|G_{2^n}|\cdot|G_{2^n}| = 2\cdot 2^{2^n-1}\cdot 2^{2^n-1} = 2^{2^n-1}$.\end{proof}

\section{Symmetry groups of the paperfolding word}

In this section we characterize symmetry groups of the paperfolding word by recurrence relations and show that the size of the group grows exponentially.

\begin{definition} The \textit{paperfolding word} $w = 110110011100100\cdots$ is defined as follows:  If $i = k\cdot 2^m$, where $k$ is odd, then
$$w_i = \begin{cases}  1, & \mbox{if } k \equiv 1 \pmod{4},  \\ 0, & \mbox{otherwise}.  \end{cases}$$ 
\end{definition}

We say that a permutation $g$ is \textit{of  type 1} if $g$ keeps the parity of each position. A permutation $g$ is \textit{of type 2} if it changes the parity of each position.

The following proposition is straightforward:

\begin{proposition}\label{pr:typeofpermutations}
Let $g_1,g_2$ be permutations of type 1 and $h_1,h_2$ be permutations of type 2. Then permutations $g_1g_2$, $g_1h_1$ and $h_1h_2$ are permutations of types 1, 2 and 1, respectively.
\end{proposition}

\begin{lemma}\label{lemma:paperfolding}
Let $w$ be a paperfolding word and $g \in G_n$, $n \geq 4$. Then $g$ is a permutation of the first or second type.
\end{lemma} 

\begin{proof}
We color positions from 1 to $n$ into 4 colors, into color 1 all positions congruent to 1 modulo 4, into color 2 all positions congruent to 2 modulo 4, and so on. Each color is present since $n \geq 4$. Note that for any factor odd positions of the word $w$ are either of colors 1 and 3, or colors 2 and 4. So there are 4 possible situations: on positions of color 1 there are only 1's, on positions of color 3 only 0's, or vice versa, or on positions of color 2 there are only 1's, on positions of color 4 only 0's, or vice versa. Consider a factor $u = a_{m+1}a_{m+2}\ldots a_{m+n}$ of a word $w$, $m+1  \equiv 1 \pmod{4}$. Then all positions of color 1 have 1 and all positions of color 3 have 0 ($u = 1\_0\_1\_0\_1\_0\_\ldots$). Let us see what colors receive the positions of the first color when $g$ is applied.

\textbf{(a)} If the positions of the first color change to the positions of each color then there is no factor $g(u)$ since $g (u)$ does not have a color with positions with only 0's.

\textbf{(b)} If the positions of the first color change to the positions of colors $1,2$ and $3$ then there are only 0's at positions of the 4th color of the factor $g(u)$. It means that at positions of the 2nd color there are only 1's. But note that there is a position of the factor $u$ not of the first color, which is translated to the position of the second color of the factor $g(u)$. It is not difficult to understand that this position could be 0 which means that there is a factor $u$ that $g(u)$ is not a factor.

The remaining cases when the positions of the first color change to the positions of the other three colors, are dealt with in the same way.

\textbf{(c)} If the positions of the first color are translated to the positions of the 1st and 3rd color, then $g (u)$ either has 0 at positions of the 2nd color and 1 at positions of the 4th color or vice versa. Assume that some positions of the 3rd color go to positions of the 2nd color. Then at positions of the 4th color of the factor $g (u)$ there are 1's. It is obvious that there is a position $m$ of another color in the factor $u$ such that the position $g(m)$ of a word $g(u)$ is of the $4$ color and $u$ has 0 at position $m$. Then $g(u)$ has 0 at position $g(m)$ of the 4th color, but it is impossible. Similarly, the positions of the 3rd color cannot be translated to the positions of the 4th color. It means that positions of the third color of the factor $u$ are translated to the positions of the 1st and 3rd colors of the factor $g(u)$.

Let $g$ is translate the positions $a,b$ of the 2nd or 4th colors of the factor $u$  to the positions of the second color of the factor $g(u)$. If we find a position $s>n$ such that $s$ is even, $a_s = 0$ and $a_{s+b-a} = 1$, then we can consider the factor $u$, where $a_s$ is at position $a$ and then $g(u)$ has both 0 and 1 at positions of the 4th color, but it is impossible. Such $s$ exists. Let $b-a = 2^mr$. If $r \equiv 3 \pmod{4}$ then we can take $s = 2^{m+1}(4t+3)$. If $r \equiv 1 \pmod{4}$ then we can take $s = 2^{m+2}(4t+3)$.

The case when the positions of the first color pass into the positions of the 2nd and 4th colors is analyzed similarly.

\textbf{(d)} If the positions of the first color are translated to the positions of the 1st and 2nd color. In this case $g(u)$ either has 0's at positions of the 4th color and 1's at positions of the 2nd color or 0's at positions of the 3rd color and 1's at positions of the 1st color. Suppose that some positions of the 3rd color are translated to positions of the 2nd color. Then $g(u)$ has 1's at positions of the 1st color   and 0's at positions of the 3rd color. But it is obvious that there is a position $m$ not of the 3rd color in the factor $u$ which is translated to the position $g(m)$ of the 3rd color in the factor $g(u)$ and  $u$ has 1 at position $m$. Then $g(u)$ we has 1 at position $g(m)$ of the 3rd color, but this is impossible. Similarly, the positions of the 3rd color cannot go to the positions of the 1st color, which means that positions of the third color go to the positions of the 3rd and 4th colors.

Suppose that there are positions $a,b$ of the 2nd or 4th colors of the factor $u$ such that $g(a)$ of the first color and $g(b)$ of the second. If we find a position $s>n$ such that $s$ is even, $a_s = 0$ and $a_{s+b-a} = 0$  then we can consider the factor $u$ where $a_s$ is at position $a$ and then $g(u)$ has both 0's and 1's at positions of the 1st and 2nd colors, but it is impossible. We show that such $s$ exists. Let $b-a = 2^mr$. If $r \equiv 1 \pmod{4}$ then we  take $s = 2^{m}(4t+3)$ where $\frac{4t+r+3}{2^j} \equiv 3 \pmod{4}$, $i$ is the maximal exponent of 2 in $4t+r+3$. If $r \equiv 3 \pmod{4}$, then we can take $s = 2^{m+2}(4t+3)$.

It means that we can translate only positions of the first color to positions of the first and second colors.

A similar reasoning apply in cases where the positions of the first color are translated to the positions of the 1st and 4th colors; in the positions of the 2nd and 3rd colors; in the positions of the 3rd and 4th colors.

These arguments imply that to the positions of any color in the factor $g(u)$ we can only translate positions of only one color in factor $u$. So the colors do not "mix". This means that positions of any color of the factor $u$ go to the positions of only one color of the factor $g(u)$. For example, if $n = 4k+1$, then there are more positions of the first color than of any other color. It means that positions of the first color translate only to the positions of the first color.

Let us assume that positions of the 1st color are translated to the positions of the 1st color. Then the positions of the 3rd color must be translated to the positions of the 3rd color. Assume the converse. Suppose that positions of the 3rd color are translated to the positions of the 2nd color. Then there are positions $a,b$ of the 2nd and 4th colors which are translated to positions of the 4th and 3rd colors, respectively. Then, for the point $c$, there is such a factor $u$ that at position $a$ we have $a_s = 0$ and at position $b$ we have $a_{s+b-a} = 1$. Then the positions of the 1st and 3rd colors are filled with 1's, and the positions of the 2nd and 4th colors are filled with 0's, but it is impossible. Similarly we show that positions of the 3rd color cannot be translated to the positions of the 4th color. This means that positions of the 3rd color should move to the positions of the 3rd color.

From this reasoning, we get that either even positions go to even positions, and odd positions go to odd positions or vice versa. It means that the permutation $g$ is a permutation of the first or second type. The lemma is proved.
\end{proof}

\begin{theorem} 
Let $w$ be the paperfolding word. We denote by $G_k^1$ and $G^2_k$ all permutations in $G_k$ of types 1 and 2, respectively. Then the symmetry group of $w$ is $G_n = G_n^1$ for odd $n$ and $G_n = G_n^1\cup G^2_n$ for even $n$ and the following recurrent formula holds:
$$G_n = \begin{cases} G^1_{[\frac{n}{2}]}\times G^1_{[\frac{n+1}{2}]}\cup G^2_{[\frac{n}{2}]}\times G^2_{[\frac{n+1}{2}]}, & \text{if n is odd}, \\ 
G_{[\frac{n}{2}]}\times G_{[\frac{n}{2}]}, & \text{if n is even}.
\end{cases}$$
\end{theorem}

\begin{remark} Here if $n$ is even, the group $G_n(w)$ is a direct product of two copies of $G_{\frac{n}{2}}$, one of them acts on even elements, the other one on odd elements. For odd $n$, the group $G_n(w)$ consists of two parts. Permutations from the first part  act as permutations from $G^1_{[\frac{n}{2}]}$ and  $G^1_{[\frac{n+1}{2}]}$ on even and odd positions, respectively. Permutations from the second part act as permutations from $G^2_{[\frac{n}{2}]}$ and  $G^2_{[\frac{n+1}{2}]}$  of permutations of type 2 on even and odd positions, respectively.  \end{remark}

\begin{proof} It is easy to see that $G_1 = S_1, G_2 = S_2, G_3 = S_3$. Consider $n \geq 4$. Let $u = w_{s+1}w_{s+2}\cdots w_{s+n}$ be a factor of $w$. It follows from Lemma \ref{lemma:paperfolding} that $g \in G_n$ is a permutation of the first or the second type. We analyze 2 cases:

\textbf{1.} Let $g \in G_n$ be a permutation of the first type acting on odd positions (1st and 3rd color) as a permutation $h$ and on even positions as $p$. Then consider an arbitrary factor $u = w_{s+1}w_{s+2}\cdots w_{s+n}$ and let $g(u) = w_{r+1}w_{r+2}\cdots w_{r+n}$. If $s \equiv 0 \pmod{4}$ then at positions of the first color of the factor $u$ we have 1's and at positions of the third color we have 0's. Consider two cases:

\textbf{1.1.} If $h$ is a permutation of the first type (translates positions of the 1st and 3th color to positions of the 1st and 3th color respectively), then $r \equiv 0 \pmod{4}$. Then $p(a_{s+2}a_{s+4}a_{s+6}\cdots) = a_{r+2}a_{r+4}a_{r+6}\cdots$. So $p(a_{\frac{s+2}{2}}a_{\frac{s+4}{2}}a_{\frac{s+6}{2}}\cdots) = a_{\frac{r+2}{2}}a_{\frac{r+4}{2}}a_{\frac{r+6}{2}}\cdots$ because $a_{2t} = a_t$. If $p \in G_{[\frac{n}{2}]}$ is permutation of first type, then $p$ is realized for some $g\in G_n$. The permutation $p$ is of the first type, because $\frac{s}2$ and $\frac{r}2$ are even. And if $p \notin G_{[\frac{n}{2}]}$, then $g \notin G_n$, because there is exist a factor $u$ that $p(a_{\frac{s+2}{2}}a_{\frac{s+4}{2}}a_{\frac{s+6}{2}}\cdots)$ is not a factor of the infinite word $w$, which means that $g(u)$ is not a factor.

\textbf{1.2.} If $h$ is a permutation of the second type (translates positions of the 1st and 3th color to positions of the 3th and 1st color respectively) then $r \equiv 2 \pmod{4}$. Then $p(a_{s+2}a_{s+4}a_{s+6}\cdots) = a_{r+2}a_{r+4}a_{r+6}\cdots$. So $p(a_{\frac{s+2}{2}}a_{\frac{s+4}{2}}a_{\frac{s+6}{2}}\cdots) = a_{\frac{r+2}{2}}a_{\frac{r+4}{2}}a_{\frac{r+6}{2}}\cdots$, because $a_{2t} = a_t$.  
If $p \in G_{[\frac{n}{2}]}$ is permutation of second type then $p$ is realized for some $g\in G_n$. The permutation $p$ is of the second type because $\frac{s}2$ is odd and $\frac{r}2$ is even. And if $p \notin G_{[\frac{n}{2}]}$, then $g \notin G_n$, because there is exist a factor $u$ that $p(a_{\frac{s+2}{2}}a_{\frac{s+4}{2}}a_{\frac{s+6}{2}}\ldots)$ is not a factor of the infinite word $w$, which means that $g(u)$ is not a factor.

We obtained that if $s \equiv 0 \pmod{4}$, then the permutation $g$ must translate odd positions to odd ones as $h$ and even positions to even ones as $p$, and these permutations must be of the same types and $p \in G_{[\frac{n}{2}]}$. We also proved that any such permutation $p \in G_{[\frac{n}{2}]}$ is realized for some $g\in G_n$.

Similarly the case of $s \equiv 2 \pmod{4}$ is analyzed and in the two remaining cases we get that $h \in G_{[\frac{n+1}{2}]}$.

We obtained that if a permutation $g$ is of the first type, then it translates even positions as a permutation $h\in G_{[\frac{n+1}{2}]}$ and odd positions as $p\in G_{[\frac{n}{2}]}$, and these permutations must be of the same type. Moreover, if $h$ and $p$ satisfy these conditions, then $g\in G_n$.

\textbf{2.} Let $g\in G_n$ be a permutation of the second type (this is possible only for even $n$) which changes the parity of each position. 
We associate a permutation $h\in G_{\frac{n}{2}}$ to the action of $g$ on odd positions, i.e. $h(12\ldots \frac{n}{2})=\frac{g(1)}{2} \frac{g(3)}{2}\ldots \frac{g(n-1)}{2}$, and we associate a permutation $p\in G_{\frac{n}{2}}$ to the action of $g$ on even positions, i.e. $p(12\ldots \frac{n}{2})=\frac{g(2)+1}{2} \frac{g(4)+1}{2}\ldots \frac{g(n)+1}{2}$.
Consider an arbitrary factor $u = a_{s+1}a_{s+2}\ldots a_{s+n}$, and let $g(u) =a_{r+1}a_{r+2}\ldots a_{r+n}$. If $s\equiv 0\pmod{4}$, then the positions of the first color of the factor $u$ are 1 and the positions of the third color are 0. Consider two cases:

\textbf{2.1.} If $h$ is a permutation of the first type (translates the positions of the 1st and 3rd color to the positions of the 2nd and 4th color respectively) then $r\equiv 3\pmod{4}$. Then  $p(a_{s+2}a_{s+4}a_{s+6}\ldots) = a_{r+1}a_{r+3}a_{r+5}\ldots$. So $p(a_{\frac{s+2}{2}}a_{\frac{s+4}{2}}a_{\frac{s+6}{2}}\ldots) = a_{\frac{r+1}{2}}a_{\frac{r+3}{2}}a_{\frac{r+5}{2}}\ldots$, since $a_{2t} = a_t$. If  $p\in G_{[\frac{n}{2}]}$ is the permutation of the second type, then $p$ is  realized for  some $g \in G_n$. The  permutation $p$ is  of  the  second  type because $\frac{s+2}{2}$ and $\frac{r+1}{2}$ are numbers of different parity. And if $p \notin G_{[\frac{n}{2}]}$, because there is exist a factor $u$ that $p(a_{\frac{s+2}{2}}a_{\frac{s+4}{2}}a_{\frac{s+6}{2}}\ldots)$ is not a factor of the infinite word $w$, which means that $g(u)$ is not a factor.

\textbf{2.2.} If $h$ is a permutation of the second type (translates the positions of the 1st and 3rd color to the positions of the 4th and 2nd color respectively) then $r\equiv 1\pmod{4}$. Then  $p(a_{s+2}a_{s+4}a_{s+6}\ldots) = a_{r+1}a_{r+3}a_{r+5}\ldots$. So $p(a_{\frac{s+2}{2}}a_{\frac{s+4}{2}}a_{\frac{s+6}{2}}\ldots) = a_{\frac{r+1}{2}}a_{\frac{r+3}{2}}a_{\frac{r+5}{2}}\ldots$, since $a_{2t} = a_t$. 
If  $p\in G_{[\frac{n}{2}]}$ is the permutation of the first type, then $p$ is realized for  some $g \in G_n$. The  permutation $p$ is  of  the  first  type because $\frac{s+2}{2}$ and $\frac{r+1}{2}$ are odd numbers. And if $p \notin G_{[\frac{n}{2}]}$ then $g\notin G_n$ because there is exist a factor $u$ that $p(a_{\frac{s+2}{2}}a_{\frac{s+4}{2}}a_{\frac{s+6}{2}}\ldots)$ is not a factor of the infinite word $w$, which means that $g(u)$ is not a factor.

We obtained that if $s\equiv 0\pmod{4}$, then the permutation $g$ must translate odd positions to even ones according to the permutation $h$,  odd positions to even ones according to $p$, these permutations must be of different types, and $p\in G_{[\frac{n}{2}]}$. We also proved that any such permutation $p\in G_{[\frac{n}{2}]}$ is realized for some $g \in G_n$.

The case $s\equiv 2 \pmod{4}$ is analyzed similarly, and in the two remaining cases we get that $h\in G_{[\frac{n+1}{2}]}$.

We  obtained that if the permutation $g$ is of the second type, then $n$ must be even and $g$ must translate odd positions to even ones as a permutation $h \in G_{[\frac{n+1}{2}]}$ and even positions to odd ones as $p \in G_{[\frac{n}{2}]}$, and these permutations must be of different types. Moreover, if $h$ and $p$ satisfy these conditions, then $g \in G_n$.

From Proposition \ref{pr:typeofpermutations} we obtain a natural operation on permutations of the first and the second type and then for odd $n$ the symmetry group consists only of permutations of type 1 and it is equal to $G_n=G_n^1=G^1_{[\frac{n}{2}]}\times G^1_{[\frac{n+1}{2}]}\cup G^2_{[\frac{n}{2}]}\times G^2_{[\frac{n+1}{2}]}$. For even $n$ the symmetry group consists of permutations of types 1 and 2 and it is equal to
$G_n = G_n^1\cup G^2_n = G^1_{[\frac{n}{2}]}\times G^1_{[\frac{n}{2}]} \cup G^2_{[\frac{n}{2}]}\times G^2_{[\frac{n}{2}]} \cup G^1_{[\frac{n}{2}]}\times G^2_{[\frac{n}{2}]} \cup G^2_{[\frac{n}{2}]}\times G^1_{[\frac{n}{2}]} = G_{[\frac{n}{2}]}\times G_{[\frac{n}{2}]}$. The theorem is proved.
\end{proof}

\begin{corollary} 
For $n \geq 3$ the following inequality holds: $|G_n| > 2^{\frac{n}{5}}$.
\end{corollary}

\begin{proof} We need to show that $|G_n| \geq 2^{\frac{n}{5}}$ for $n \geq 3$. We will prove that the number of permutations of the first type is at least $2^{\frac{n}{5}}$. For $n \geq 3$ this is true because for $n = 3$ there are two permutations of the first type  ([1,2,3],[3,2,1]), for $n = 4$ there are at least 4 permutations of the first type ([1,2,3,4],[3,2,1,4],[1,4,3,2], [3,4,1,2]) and for $n = 5$ there are at least 2 permutations of the first type  ([1,2,3,4,5],[5,2,3,4,1]). For bigger $n$ it can be proved by induction. It follows from the 1st case that there are at least $2^{\frac{[\frac{n}{2}]}{5}}\cdot2^{\frac{[\frac{n+1}{2}]}{5}} = 2^{\frac{n}{5}}$  permutations of the first type, because each permutations $g$ of the first type consists of permutations of the first type $h \in G_{[\frac{n+1}{2}]}$ and $p \in G_{[\frac{n}{2}]}$ belong to $G_n$. That is what we needed to prove. 
\end{proof}

\section{Symmetry groups of Toeplitz words}

In this section we characterize symmetry groups of a subclass of Toeplitz words and show that symmetry groups of Toeplitz words are quite diverse.

\begin{definition} Let ? be a letter not in  
$\Sigma$. For a word $u \in \Sigma(\Sigma \cup \{?\})^{*}$, let $$T_0(u) = ?^{\omega}, \quad T_{i+1}(u) = F_u(T_i(u)),$$ where $F_u(w)$, defined for any $w \in (\Sigma \cup \{?\})^{\omega}$, is the word obtained from $u^{\omega}$ by replacing the sequence of all occurrences of ? by $w$.
Clearly, $$T(u) = \lim\limits_{i \to \infty} T_i(u) \in \Sigma^{\omega}$$ is well defined, and it is referred to as the \emph{Toeplitz word} determined by the pattern $u$.
\end{definition}

\begin{example}
The paperfolding and the period-doubling words are Toeplitz words determined by patterns $1?0?$ and $010?$, respectively.  
\end{example}

Let $u$ be a pattern with one space (or symbol ?), $|u| = k \geq 4$ and all letters of $u$ distinct. Let $w$ be the Toeplitz word determined by $u$, and $G_n$ be the symmetry group of $w$. We divide the positions of  $w$ into $k$ groups of positions congruent modulo $k$. On one of these groups there are spaces in the word $T_1(u)$; we let $T_{space}$ denote this group. 
The positions of any factor of $w$ are also divided into $k$ corresponding groups. We first prove two 
lemmas concerning these groups. 

\begin{lemma}\label{lemma:Toeplitz}
Any $\sigma \in G_n$ translates positions from the same group to positions from the same group (probably distinct from the initial group). 
\end{lemma}

\begin{proof}
Assume the converse. Let $A$ be a group whose positions move to the positions of different groups $B$ and $C$. Then consider any factor $v_x$ of the word $w$ of length $n$ in which the letter $x$ stands at each position from the group $A$. Then $v' = \sigma(v)$ is a factor of $w$. So the word $v'$ has the letter $x$ at some positions of the groups $B$ and $C$. Then for any occurrence of $v'$ in $w$, one of the groups $B$ and $C$ of $v'$  corresponds to $T_{space}$ in $T_1(u)$.  

Since $|u|\geq 4$, we can choose three words $v_x,v_y,v_z$ in which all positions of $A$ are filled with three distinct letters $x, y$, and $z$, respectively. Then for at least two of the three factors $\sigma(v_x),\sigma(v_y)$ and $\sigma(v_z)$, the group $T_{space}$ corresponds to either $B$ or $C$; say in $\sigma(v_x)$ and $\sigma(v_y)$ the group $B$ corresponds to $T_{space}$. Hence only one letter can correspond to the group $C$ in  $w$. But in the words $\sigma(v_x)$ and $\sigma(v_y)$ the positions at the group $C$ have the letters $x$ and $y$. A contradiction.
\end{proof}

\begin{lemma}\label{lemma:Toeplitz-2}
Any $\sigma \in G_n$ acts on the groups as a cyclic shift or as the identity permutation.
\end{lemma}

\begin{proof}
Lemma \ref{lemma:Toeplitz} implies that groups are not split. Since all the letters in $u$ are distinct, the distance between any two groups does not change. This means that the order of the groups remains the same. Also note that if $n$ is not divisible by $k$, then the groups are not of the same size. Then $\sigma \in G_n$ acts on the groups as the identity permutation.
\end{proof}

\begin{theorem} Let $n = ak+b$, where $0\leq b \leq k-1$. 
Then $$G_n = \begin{cases} 
G_{a+1}^b\times G_{a}^{k-b}, & \text{if b} \neq 0, \\ G_{a}^{k}\times{\mathbb{Z}/k\mathbb{Z}}, & \text{if b} = 0.
\end{cases}$$
\end{theorem}

\begin{proof} 
Let $v$ be a factor of length $n$. Then $u$ has $b$ groups with $a+1$ positions and $k-b$ groups with $a$ positions; we let $A_1,\ldots, A_k$ denote these groups. Suppose first that $n$ is not divisible by $k$. Then from the previous lemmas we get that $\sigma \in G_n$ translates the group $A_i$ to itself for each $i$. Let $\pi_1,\ldots, \pi_k$ be restrictions of $\sigma$ to the groups $A_1,\ldots, A_k$, respectively. Then $\pi_i \in G_a$ or $\pi_i \in G_{a+1}$, depending on the size of  $A_i$. Indeed, consider any factor $v$ for which the group $A_i$ corresponds to $T_{space}$. Let $v'$ be the word consisting of the letters of the word $v$ located at positions from $A_i$. From the definition of Toeplitz words, it follows that the scattered subword at positions $T_{space}$ in the word $w$ is $w$. This means that $v'$ can be any factor of $w$ of corresponding length. Since $\sigma(v)$ translates positions from $A_i$ to  positions from $A_i$, which corresponds to $T_{space}$, the word $v'$ is translated to the word $\pi_i(v')$, which is located at positions $T_{space}$ of the factor $\sigma(v)$. Then $\pi_i(v')$ is a factor of $w$. So $\pi_i \in G_a$ or $\pi_i \in G_{a+1}$.

Let us prove that if all permutations $\pi_i$ are from $G_a$ or $G_{a+1}$, then $\sigma \in G_n$. Consider any factor $v$ of $w$. Let $A_i$ be the group corresponding to $T_{space}$ and $v'$ be the word  consisting of the letters of the word $v$ located at positions from $A_i$. Since $\pi_i \in G_a$ (or $G_{a+1}$), we have that $\pi_i(v')$ is a factor of $w$. So we can find the word $\pi_i(v')$ at positions $T_{space}$ of $w$. Then it follows that $\sigma(v)$ is a factor of $w$. 
So in this case we get that $G_n = G_{a+1}^b\times G_{a}^{k-b}$, because the factor of length $n$ has $b$ groups with $a+1$ positions and $k-b$ groups with $a$ positions.

If $b = 0$, then a shift can be applied to the permutation $\sigma$, and there are $k$ possible shifts. So,  in this case we have $G_n = G_{a}^{k}\times{\mathbb{Z}/k\mathbb{Z}}$. The theorem is proved.
\end{proof}

\begin{corollary}
Consider the function $f(n) = |G_n(w)|$. Let $n = ak+b$, where $0\leq b \leq k-1$. Then 
$$f(ak+b) = 
\begin{cases}
f(a)^{k-b}f(a+1)^b & \text{if} \,\, b \neq 0, \\ 
 f(a)^{k}k & \text{if} \,\, b = 0. 
\end{cases}$$
\end{corollary}

\begin{remark} 
The corollary shows that $f(n)=|G_n|$ can fluctuate. For example, by induction one can show that if $n = ak+b$, $b\neq 0$ and $a$ does not contain $0$ and $k-1$ in the $k$-ary expansion, then $f(n) = 1$. For $n = k^s$ we have  $f(n)=k^{\frac{k^s-1}{k-1}}$. 
\end{remark}

If there is more than one space in $u$, the groups can behave differently. Contrary to the paperfolding word, for which the symmetry groups are large, the following words have trivial groups: 

\begin{proposition}\label{pr:toeplitz-1}
The symmetry groups of $T(1??23)$ and $T(12??34)$ 
are equal to $\mathrm{Id}$ for each $n$. 
\end{proposition}

\begin{proof} 
First we prove the statement for $T(1??23)$.

Consider $\sigma \in G_n$. Assume that $\sigma \neq \mathrm{Id}$. Then there are $1 \leq a < a+k \leq n$ such that $\sigma(a) = r$, $ \sigma (a+k) = r-1$. We will prove that there is a factor $v$ such that $\sigma(v)$ is not a factor of $w=T(1??23)$.

Note that $w_{5k+1} = 1, w_{5k+4} = 2, w_{5k+5} = 3$. We call a pair $(m,n)\in\Sigma^2$ forbidden if we do not have $w_s = m, w_{s+1} = n$. It is not difficult to see that the pairs  $(2,1)$, $(3,3)$ are forbidden.

The idea of the proof is that if $v_a = n$ and $v_{a+k} = m$ for some factor $v$ of $w$, then $\sigma(v)_{r} = n$ and $\sigma(v)_{r-1} = m$. But it is impossible for any forbidden pair $(m,n)$. Consider 5 cases.

\begin{itemize}
 \item $k \equiv 0 \pmod{5}$. Then there is a factor $v$ such that $v_a = w_{5r+5} = 3$ and $v_{a+k} = w_{5r+5+k} = 3$. But the $(3,3)$ pair is forbidden; a contradiction.

\item $k \equiv 1 \pmod{5}$. Then there is a factor $v$ such that $v_a = w_{5r+1} = 1$ and $v_{a+k} = w_{5r+1+k} = 2$. But the $(2,1)$ pair is forbidden; a contradiction.

\item $k \equiv 2 \pmod{5}$. Then there is a factor $v$ such that $v_a = w_{5r+3} = 3$ and $v_{a+k} = w_{5r+3+k} = 3$. But the $(3,3)$ pair is forbidden; a contradiction.

\item $k \equiv 3 \pmod{5}$. Then there is a factor $v$ such that $v_a = w_{5r+2} = 3$ and $v_{a+k} = w_{5r+2+k} = 3$. But the $(3,3)$ pair is forbidden; a contradiction.

\item $k \equiv 4 \pmod{5}$. Then there is a factor $v$ such that $v_a = w_{5r+3} = m$ and $v_{a+k} = w_{5r+3+k} = n$, where $(m,n)$ is a forbidden pair. It is true because the word $w$ on spaces is also the word $w$ ($w_2w_3w_7w_8w_12w_13\cdot = w$). So we can consider the permutation on the word on spaces and reduce $r$. Continuing this process we will reduce $r$ and get one of the previous cases. So we get a contradiction. 

\end{itemize}
So, we proved that $\sigma = \mathrm{Id}$.

\bigskip

Now we prove the statement for  $T(12??34)$.

Consider $\sigma \in G_n$. Assume that $\sigma \neq \mathrm{Id}$. Then there are $1 \leq a < a+k \leq n$ such that $\sigma(a) = r \sigma (a+k) = r-1$. We prove that there is a factor $v$ that $\sigma(v)$ is not a factor of $w=T(12??34)$.

Note that $w_{6k+1} = 1, w_{6k+2} = 1, w_{6k+5} = 3, w_{6k+6} = 4$. We call a pair $(m,n)$ forbidden if the following does not hold: $w_s = m, w_{s+1} = n$. It is not difficult to understand that the following pairs are forbidden: (1,1), (1,3), (1,4), (2,2), (2,4), (3,1), (3,2), (3,3), (4,2), (4,4).

The idea of the proof is that if $v_a = n$ and $v_{a+k} = m$, then $\sigma(v)_r = m$ and $\sigma(v)_{r-1} = m$, which is not the case for any forbidden pair $(m,n)$. Consider 6 cases.

\begin{itemize}
\item $k \equiv 0 \pmod{6}$. Then there is a factor $v$ such that $v_a = w_{6r+1} = 1$ and $v_{a+k} = w_{6r+1+k} = 1$. But the $(1,1)$ pair is forbidden; a contradiction.

\item $k \equiv 1 \pmod{6}$. Then there is a factor $v$ such that $v_{a} = w_{6r+4} = 2$ and $v_{a+k} = w_{6r+4+k} = 3$. But the $(3,2)$ pair is forbidden; a contradiction.

\item $k \equiv 2 \pmod{6}$. Then there is a factor $v$ such that $v_{a} = w_{6r+4} = 4$ and $v_{a+k} = w_{6r+4+k} = 4$. But the $(4,4)$ pair is forbidden; a contradiction.

\item $k \equiv 3 \pmod{6}$. Then there is a factor $v$ such that $v_a = w_{6r+4} = 4$ and $v_{a+k} = w_{6r+4+k} = 1$. But the $(1,4)$ pair is forbidden; a contradiction.

\item $k \equiv 4 \pmod{6}$. Then there is a factor $v$ such that $v_{a} = w_{6r+2} = 2$ and $v_{a+k} = w_{6r+2+k} = 4$. But the $(4,2)$ pair is forbidden; a contradiction.

\item $k \equiv 5 \pmod{6}$. Then there is a factor $v$ such that $v_{a} = w_{6r+4} = m$ and $v_{a+k} = w_{6r+4+k} = n$ where $(m,n)$ is a forbidden pair. It is true because the word $w$ on spaces is also the word $w$ ($w_3w_4w_9w_10w_15w_16\cdot = w$). So we can consider the permutation on the word on spaces and reduce $r$. Continuing this process we will reduce $r$ and get one of the previous cases. So we get a contradiction.

\end{itemize}
So, we proved that $\sigma = \mathrm{Id}$.
\end{proof}

\section{Conclusions and open problems}

In this paper, we introduced and studied a new notion of symmetry groups of infinite words. We remark that all the words considered in the paper have similar properties: linear complexity, rich combinatorial structure, all of them except for Arnoux-Rauzy words are automatic, but their symmetry groups are completely different. It would be interesting to understand in general what properties of a word make its symmetry group large. An interesting direction of future research is generalising results from Section \ref{sec:sequencies}, answering the following question: Which infinite sequences $(G_n)_{n\geq 1}$, $G_n\leq S_n$ are symmetry groups of infinite words?

\end{document}